\documentclass[11pt,times,twoside]{article}

\usepackage{graphicx,latexsym,euscript,makeidx,color}
\usepackage{amsmath,amsfonts,amssymb,amsthm,mathrsfs}
\usepackage{enumerate}

\usepackage[colorlinks,linkcolor=blue,anchorcolor=green,citecolor=red]{hyperref}%when print, use the next package
\usepackage{geometry}
\def\5n{\negthinspace \negthinspace \negthinspace \negthinspace \negthinspace }
\def\4n{\negthinspace \negthinspace \negthinspace \negthinspace }
\def\3n{\negthinspace \negthinspace \negthinspace }
\def\2n{\negthinspace \negthinspace }
\def\1n{\negthinspace }
\def\ms{\medskip}

\def\no{\noindent}        \def\q{\quad}                      
    \def\qq{\qquad}

            \def\({\Big (}
                  \def\){\Big )}
          \def\[{\Big[}
           \def\]{\Big]}

\def\bde{\begin{definition}\label}    \def\ede{\end{definition}}
\def\be{\begin{equation}}
\def\bel{\begin{equation}\label}      \def\ee{\end{equation}}
\def\bt{\begin{theorem}\label}        \def\et{\end{theorem}}
\def\bc{\begin{corollary}\label}      \def\ec{\end{corollary}}
\def\bl{\begin{lemma}\label}          \def\el{\end{lemma}}
\def\bp{\begin{proposition}\label}    \def\ep{\end{proposition}}
\def\bas{\begin{assumption}\label}    \def\eas{\end{assumption}}
\def\br{\begin{remark}\label}         \def\er{\end{remark}}
\def\bex{\begin{example}\label}       \def\ex{\end{example}}
\def\ba{\begin{array}}                \def\ea{\end{array}}
\def\ben{\begin{enumerate}}           \def\een{\end{enumerate}}

\newtheorem{theorem}{Theorem}[section]
\newtheorem{definition}[theorem]{Definition}
\newtheorem{proposition}[theorem]{Proposition}
\newtheorem{corollary}[theorem]{Corollary}
\newtheorem{lemma}[theorem]{Lemma}
\newtheorem{remark}[theorem]{Remark}
\newtheorem{example}[theorem]{Example}

\makeatletter
   
   \@addtoreset{equation}{section}
\makeatother

\allowdisplaybreaks[4]
\begin{document}

\title{\bf  Backward doubly stochastic differential equations with random coefficients and quasilinear stochastic PDEs}

\author
{\textbf{Jiaqiang Wen}$^{1,2}$, \textbf{Yufeng Shi}$^{1,3,\dagger}$ \vspace{3mm}\\
\normalsize{$^{1}$Institute for Financial Studies and School of Mathematics,}\\
\normalsize{Shandong University, Jinan 250100, China}\\
\normalsize{$^{2}$Department of Mathematics, Southern University of Science and Technology,} \\
\normalsize{Shenzhen, Guangdong, 518055, China}\\
\normalsize{$^{3}$School of Statistics, Shandong University of Finance and Economics,}\\
\normalsize{ Jinan, Shandong, 250014, China}\\
}

\date{}
\renewcommand{\thefootnote}{\fnsymbol{footnote}}

\footnotetext[0]{This work is supported by National Natural Science Foundation of China (Grant Nos. 11871309, 11371226, 11671229, 11071145, 11526205, 11626247 and 11231005), the Foundation for Innovative Research Groups of National Natural Science Foundation of China (Grant No. 11221061) and the 111 Project (Grant No. B12023).}

\footnotetext[2]{Corresponding author. E-mail addresses: wenjq@sustc.edu.cn (J. Wen), yfshi@sdu.edu.cn (Y. Shi).}

\maketitle

%--------------------------------------------------------------------------------------------------------

\no\bf Abstract. \rm
In this paper, by virtue of Malliavin calculus, we establish a relationship between backward doubly stochastic differential equations with random coefficients and quasilinear stochastic PDEs, and thus extend the well-known nonlinear stochastic Feynman-Kac formula of Pardoux and Peng \cite{Peng} to non-Markovian case.

\ms

\no\bf Key words. \rm
Backward doubly stochastic differential equations;
Stochastic partial differential equations;
Nonlinear stochastic Feynman-Kac formula.

\ms

\no\bf AMS subject classifications. \rm 60H10; 60H20.
%49J15, 49N10, 49N35, 93B05

\section{Introduction}

It is well-known that Peng \cite{Peng91} and Pardoux and Peng \cite{Peng92} established the so-called nonlinear Feynman-Kac formula
by virtue of nonlinear backward stochastic differential equations (BSDEs in short) introduced by Pardoux and Peng \cite{Peng90},
that is they constructed a connection between quasilinear parabolic PDEs and Markovian forward-backward stochastic differential equations
(FBSDEs in short). Later, Pardoux and Peng \cite{Peng} introduced the so-called backward doubly stochastic differential equations (BDSDEs in short)
in order to give a probabilistic representation of solutions to a class of systems of quasilinear parabolic stochastic partial differential equations (SPDEs in short).
They established the well-known nonlinear stochastic Feynman-Kac formula by virtue of BDSDEs.

In order to illustrate the nonlinear stochastic Feynman-Kac formula, let us be more specific.
Let $\{W_{t};0\leq t\leq T\}$ and $\{B_{t};0\leq t\leq T\}$ be two mutually independent standard Brownian motions
with values in $\mathbb{R}^{d}$ and in $\mathbb{R}^{l}$, respectively.
For each $(t,x)\in [0,T]\times \mathbb{R}^{d}$, let $\{X^{t,x}_{s},t\leq s\leq T\}$ be
the solution of the following stochastic differential equation (SDE in short):
\begin{equation*}
  X^{t,x}_{s} = x + \int_t^s b(X^{t,x}_{r}) dr + \int_t^s \sigma(X^{t,x}_{r}) dW_{r}, \ \ t\leq s\leq T.
\end{equation*}
Let $\{ (Y^{t,x}_{s},Z^{t,x}_{s});t\leq s\leq T \}$, which is $\sigma(W_{r}-W_{t}; t\leq r\leq s)\vee \sigma(B_{r}-B_{s}; s\leq r\leq T)$-measurable,
be the solution of the well-posed BDSDE:
\begin{equation}\label{17}
 \begin{split}
  Y^{t,x}_{s} =& h(X^{t,x}_{T}) + \int_s^T f(r,X^{t,x}_{r},Y^{t,x}_{r},Z^{t,x}_{r}) dr \\
               &+ \int_s^T g(r,X^{t,x}_{r},Y^{t,x}_{r},Z^{t,x}_{r}) dB_{r} - \int_s^T Z^{t,x}_{r}dW_{r}, \ \ t\leq s\leq T,
 \end{split}
\end{equation}
where $f$ and $g$ are deterministic functions. The $dW$ integral is a forward It\^{o} integral and the $dB$ integral is a backward It\^{o} integral.
As we know, under enough smoothness assumptions on $b, \sigma, f$ and $g$,
there exists a random field $\{(u(t,x);0\leq t\leq T\}$, which is $\sigma(B_{r}-B_{t};t\leq r\leq T)$-measurable,
such that $Y^{t,x}_{t}=u(t,x)$ and $Z^{t,x}_{t}=\sigma(x) \nabla u(t,x)$.
Here, $\{ u(t,x);(t,x)\in [0,T]\times \mathbb{R}^{d} \}$ is a solution of the following system of backward SPDEs:
\begin{equation*}
 \begin{split}
  u(t,x) =& h(x) + \int_t^T \big[\mathcal{L}u(s,x) + f(s,x,u(s,x),(\nabla u \sigma)(s,x))\big] ds \\
          &+ \int_t^T g(s,x,u(s,x),(\nabla u \sigma)(s,x)) dB_{s}, \ \ 0\leq t\leq T,
 \end{split}
\end{equation*}
where $u$ takes values in $\mathbb{R}^{k}$ and $\mathcal{L}u=(Lu_{1},...,Lu_{k})^{\ast}$ with
\begin{equation*}
  L=\frac{1}{2}\sum_{i,j=1}^{d} (\sigma \sigma^{\ast})_{ij}(t,x)\frac{\partial^{2}}{\partial x_{i}\partial x_{j}}
    + \sum_{i=1}^{d} b_{i}(t,x) \frac{\partial}{\partial x_{i}}.
\end{equation*}
This relation permits us to solve the above type of BDSDEs by SPDEs.
Conversely we can also use BDSDEs to solve SPDEs.
This result is summarized as the nonlinear stochastic Feynman-Kac formula.
So far, both the theory and applications of BDSDEs have been paid intensive attention.
Boufoussi, Casteren and  Mrhardy \cite{Boufoussi} investigated the generalized BDSDEs and SPDEs with nonlinear Neumann boundary conditions.
Diehl and Friz \cite{Friz} established the connection between BSDEs with rough drivers and BDSDEs.
A stochastic maximum principle for backward doubly stochastic control systems was obtained in Han, Peng and Wu \cite{Han}.
Peng and Shi \cite{Shi3} studied a type of time-symmetric forward-backward doubly stochastic differential equations.
Shi, Gu and Liu \cite{Shi} gave a comparison theorem for BDSDEs.
For other recent developments on BDSDEs, we refer to the works of
Aman \cite{Aman12,Aman13}, Hu, Matoussi, and Zhang \cite{Huy},
Wen and Shi \cite{Wen}, Zhang and Shi \cite{Zhanglq}, etc.

During the past two decades many efforts have been made to extend the nonlinear (stochastic) Feynman-Kac formula to non-Markovian situation.
Ma, Yin and Zhang \cite{Zhang} extended the nonlinear Feynman-Kac formula to the random coefficient case under the following assumption
$$Y_{t}=u(t,X_{t}), \ \ \forall t\in [0,T], \ P-a.s.$$
See \cite{Zhang} for detailed discussion.
Peng and Wang \cite{Peng16} established a nonlinear Feynman-Kac formula in non-Markovian case
by using the functional It\^{o}/path-dependent calculus.
To the best of our knowledge, a few works have been done for nonlinear stochastic Feynman-Kac formula in non-Markovian situation,
especially when the coefficients are random. In fact, if the coefficients are random, it is very difficult to find a proper probabilistic representation for the solutions of PDEs (or SPDEs).

In this paper, we consider the BDSDE (\ref{17}) with random coefficients, that is,
the coefficients $f$ and $g$ are $\sigma(B_{r}-B_{t};t\leq r\leq T)$-measurable processes.
By virtue of Malliavin calculus, we establish a relationship between this type of BDSDEs and backward quasilinear SPDEs,
thus extending the nonlinear stochastic Feynman-Kac formula of Pardoux and Peng \cite{Peng}
(also the nonlinear Feynman-Kac formula of Pardoux-Peng \cite{Peng92}) to non-Markovian situation.

This paper is organized as follows.
In Section 2, some preliminary results are presented.
Section 3 is devoted to establish the regularity of the solutions of BDSDEs.
We relate the BDSDEs to a system of backward quasilinear SPDEs in Section 4.

\section{Preliminaries}

In this section, we present some preliminaries about BDSDEs and Malliavin derivative.
The Euclidean norm of a vector $x\in\mathbb{R}^{k}$ will be denoted by $|x|$, and for a $k\times d$ matrix $A$, we define $\| A \|=\sqrt{Tr AA^{\ast}}$.

\subsection{BDSDEs}

Throughout this paper, let $(\Omega,\mathcal{F},P)$ be a probability space and $T>0$ be a fixed terminal time.
Let $\{W_{t};0\leq t\leq T\}$ and $\{B_{t};0\leq t\leq T\}$ be two mutually independent standard Brownian motion processes defined on $(\Omega,\mathcal{F},P)$,
 with values in $\mathbb{R}^{d}$ and in $\mathbb{R}^{l}$, respectively.
Let $\mathcal{N}$ denote the class of $P$-null sets of $\mathcal{F}$.
For each $t\in [0,T]$, we define
 \begin{equation*}
 \mathcal{F}_{t} :=  \mathcal{F}_{t}^{W} \vee \mathcal{F}_{t,T}^{B},
 \end{equation*}
 where for any process $\{\eta_{t}\}, \ \mathcal{F}_{s,t}^{\eta}=\sigma\{\eta_{r}-\eta_{s};s\leq r\leq t\}\vee \mathcal{N}$ and
 $\mathcal{F}_{t}^{\eta}=\mathcal{F}_{0,t}^{\eta}$.
Note that the collection $\{ \mathcal{F}_{t}; t\in[0,T] \}$ is neither increasing nor decreasing, and it does not constitute a filtration.

For any $n \in \mathbb{N}$, let $M^{2}(0,T; \mathbb{R}^{n})$ denote the set of (classes of $dP\times dt$ a.e. equal)
$n$-dimensional jointly measurable stochastic processes  $\{ \varphi_{t}; t\in[0,T] \}$ which satisfy:\\
(i) \ \ $\|\varphi \|_{M^{2}}^{2} :=E\int_0^T |\varphi_{t}|^{2} dt<\infty$;\\
(ii) \ $\varphi_{t}$ is $\mathcal{F}_{t}$-measurable, for any $t\in[0,T].$\\
Similarly, we denote by $S^{2}(0,T; \mathbb{R}^{n})$ the set of $n$-dimensional continuous stochastic processes, which satisfy:\\
(i) \ \ $\|\varphi \|_{S^{2}}^{2} :=E(\sup\limits_{0\leq t\leq T} |\varphi_{t}|^{2})<\infty$;\\
(ii) \ $\varphi_{t}$ is $\mathcal{F}_{t}$-measurable, for any $t\in[0,T].$\\
Let
\begin{equation*}\label{}
 \begin{split}
   &f:\Omega\times [0,T]\times  \mathbb{R}^{k}\times \mathbb{R}^{k\times d}\longrightarrow  \mathbb{R}^{k},\\
   &g:\Omega\times [0,T]\times  \mathbb{R}^{k}\times \mathbb{R}^{k\times d}\longrightarrow  \mathbb{R}^{k\times l},
 \end{split}
\end{equation*}
be jointly measurable such that for any $(y,z)\in \mathbb{R}^{k}\times \mathbb{R}^{k\times d}$,
\begin{equation*}
  f(\cdot,y,z)\in M^{2}(0,T; \mathbb{R}^{k}), \ \ \ \ g(\cdot,y,z)\in M^{2}(0,T; \mathbb{R}^{k\times d}).
\end{equation*}
Moreover, we assume that there exist constants $c > 0$ and $0 <\alpha < 1$ such that for any
$(\omega,t)\in\Omega\times [0,T], \ (y_,z),(y',z')\in \mathbb{R}^{k}\times \mathbb{R}^{k\times d}$,
\begin{equation*} (H1) \ \ \ \ \
 \begin{cases}
   |f(t,y,z)-f(t,y',z')|^{2}\leq c(|y-y'|^{2} + \|z-z'\|^{2});\\
   \|g(t,y,z)-g(t,y',z')\|^{2}\leq c|y-y'|^{2} + \alpha \|z-z'\|^{2}.\\
 \end{cases}
\end{equation*}
There exists $C>0$ such that for all $(\omega,t,y,z)\in \Omega\times [0,T]\times \mathbb{R}^{k}\times \mathbb{R}^{k\times d}$,
\begin{equation*} (H2) \ \ \ \ \ \ \ \ \ \ \ \
   gg^{\ast}(t,y,z)\leq z z^{\ast} + C (\| g(t,0,0)\|^{2} + |y|^{2} )I. \ \ \ \ \ \ \ \ \ \
\end{equation*}
Given $\xi\in L^{2}(\Omega,\mathcal{F}_{T},P;\mathbb{R}^{k})$, we shall consider the following BDSDE:
\begin{equation}\label{4}
\begin{split}
   Y(t)=&\xi + \int_t^T f(s,Y(s),Z(s)) ds \\
        & + \int_t^T g(s,Y(s),Z(s)) dB_{s} - \int_t^T Z(s) dW_{s}, \ \ t\in[0,T].
\end{split}
\end{equation}
Note that the integral with respect to $\{ B_{t} \}$ is a ``backward It\^{o} integral''
and the integral with respect to $\{ W_{t} \}$ is a standard forward It\^{o} integral.
These two types of integrals are particular cases of the It\^{o}-Skorohod integral, see Nualart and Pardoux \cite{Nualart}.
\begin{proposition}[Pardoux and Peng \cite{Peng}, Theorem 1.1]\label{5}
Under the condition (H1), BDSDE (\ref{4}) has a unique solution
\begin{equation*}
  (Y,Z)\in S^{2}(0,T; \mathbb{R}^{k})\times M^{2}(0,T; \mathbb{R}^{k\times d}).
\end{equation*}
\end{proposition}

\begin{proposition}[Pardoux and Peng \cite{Peng}, Theorem 1.4]\label{8}
Assume (H1) and (H2) hold, and for some $p\geq 2$, $\xi\in L^{p}(\Omega,\mathcal{F}_{T},P;\mathbb{R}^{k})$ and
$ E\int_0^T (|f(t,0,0)|^{p}+ \|g(t,0,0)\|^{p}) dt <\infty.$
Then
\begin{equation*}
  E\bigg[\sup_{0\leq t\leq T}|Y_{t}|^{p} + \big( \int_0^T \|Z_{t}\|^{2} dt \big)^{p/2}\bigg]< \infty.
\end{equation*}
\end{proposition}

\subsection{Malliavin derivative}

Let us recall the notion of derivation on Wiener space.
The following definition comes from Pardoux and Peng \cite{Peng},
and for more information about Malliavin calculus we refer the readers to  Nualart \cite{Nualart06}.

For any random variable $F$ of the form:
\begin{equation*}
  F=f\big(W(h_{1}),...,W(h_{n});B(k_{1}),...,W(k_{p})\big),
\end{equation*}
where $f\in C^{\infty}_{b}(\mathbb{R}^{n+p}),h_{i}\in L^{2}([0,T];\mathbb{R}^{d}),k_{j}\in L^{2}([0,T];\mathbb{R}^{l})$,
and
\begin{equation*}
  W(h_{i})= \int_0^T (h_{i}(t),dW_{t}), \ \ \    B(k_{j})= \int_0^T (k_{j}(t),dB_{t}), \ \ \ i=1,...,n;j=1,...,p.
\end{equation*}
To such a random variable $F$, we associate a ``derivated process'' $\{ D_{t}F; t\in[0,T] \}$ defined as
\begin{equation}\label{0}
 D_{t}F:=\sum_{i=1}^{n}f'_{i}\big(W(h_{1}),...,W(h_{n});B(k_{1}),...,B(k_{p})\big)h_{i}, \ \ \ t\in[0,T].
\end{equation}
For such an $F$, we define its $1,2$-norm as:
\begin{equation*}
  \| F \|_{1,2}^{2}=E(F^{2}) + E\int_0^T |D_{t}F|^{2} dt.
\end{equation*}
We denote by $\mathbb{S}$ the set of random variables of the above form, and define $\mathbb{D}^{1,2}$ the Sobolev space:
\begin{equation*}
  \mathbb{D}^{1,2}:=\overline{\mathbb{S}}^{\| \cdot \|_{1,2}}.
\end{equation*}
The ``derivation operator'' $D_{\mathbf{.}}$ extends as an operator from $\mathbb{D}^{1,2}$ into $L^{2}(\Omega;L^{2}$ $(\Omega\times [0,T];\mathbb{R}^{d}))$.
\begin{example}
Suppose $F\in\mathbb{D}^{1,2}$ is of the form $F=f\big(B(k_{1}),...,B(k_{p})\big)$. Then from the definition, we have $D_{t}F=0, \ t\in[0,T].$
\end{example}

\section{Regularity of the solution of BDSDEs}

Let us first recall some notations from Pardoux and Peng \cite{Peng92}.
$C^{k}(\mathbb{R}^{p}; \mathbb{R}^{q}),$ $C_{b}^{k}(\mathbb{R}^{p}$; $\mathbb{R}^{q}),C_{p}^{k}(\mathbb{R}^{p};\mathbb{R}^{q})$ will denote,
respectively, the set of functions of class $C^{k}$ from $\mathbb{R}^{p}$ into $\mathbb{R}^{q}$,
the set of those functions of class $C_{b}^{k}$ whose partial derivatives of order less then or equal to $k$ are bounded,
(and hence the function itself grows at most linearly at infinity),
and the set of those functions of class $C_{p}^{k}$ which,
together with all their partial derivatives of order less then or equal to $k$, grow at most like a polynomial function of the variable $x$ at infinity.

We are given $b\in C^{3}_{b}(\mathbb{R}^{d};\mathbb{R}^{d})$ and $\sigma\in C^{3}_{b}(\mathbb{R}^{d};\mathbb{R}^{d\times d})$,
and for each $t\in[0,T), \ x\in \mathbb{R}^{d}$, we denote by $\{X^{t,x}_{s},t\leq s\leq T\}$ the unique strong solution of the following SDE:
\begin{equation}\label{1}
  X^{t,x}_{s} = x + \int_t^s b(X^{t,x}_{r}) dr + \int_t^s \sigma(X^{t,x}_{r}) dW_{r}, \q t\leq s\leq T.
\end{equation}
It is well known that the random field $\{ X^{t,x}_{s};0\leq t\leq s\leq T, x\in\mathbb{R}^{d} \}$ has a version which is a.s. of class $C^{2}$ in $x$,
the function and its derivatives being a.s. continuous with respect to $(t,s,x)$.

Moreover, for each $(t, x)$,
\begin{equation*}
  \sup_{t\leq s\leq T} \bigg(|X^{t,x}_{s}| + |\nabla X^{t,x}_{s}| + |D^{2}X^{t,x}_{s}|\bigg) \in \bigcap_{p\geq 1}L^{p}(\Omega),
\end{equation*}
where $\nabla X^{t,x}_{s}$ denotes the matrix of first order derivatives of $X^{t,x}_{s}$ with respect to $x$
and $D^{2}X^{t,x}_{s}$ the tensor of second order derivatives.

Now the coefficients $f$ and $g$ of BDSDE
\begin{equation*}\label{}
 \begin{split}
   f:\Omega\times [0,T]\times \mathbb{R}^{d}\times \mathbb{R}^{k}\times \mathbb{R}^{k\times d}&\longrightarrow  \mathbb{R}^{k},\\
   g:\Omega\times [0,T]\times \mathbb{R}^{d}\times \mathbb{R}^{k}\times \mathbb{R}^{k\times d}&\longrightarrow  \mathbb{R}^{k\times l},
 \end{split}
\end{equation*}
will be the form
\begin{equation}\label{12}
 \begin{split}
   f(\omega,t,x,y,z)&= \overline{f}(\overleftarrow{B}(\omega,t),x,y,z),\\
   g(\omega,t,x,y,z)&= \overline{g}(\overleftarrow{B}(\omega,t),x,y,z),\\
 \end{split}
\end{equation}
where $\overline{f}$ and $\overline{g}$ are functions from
\begin{equation*}\label{}
   \overline{f}:\mathbb{R}^{l}\times \mathbb{R}^{d}\times \mathbb{R}^{k}\times \mathbb{R}^{k\times d}\longrightarrow  \mathbb{R}^{k},\ \ \
   \overline{g}:\mathbb{R}^{l}\times \mathbb{R}^{d}\times \mathbb{R}^{k}\times \mathbb{R}^{k\times d}\longrightarrow  \mathbb{R}^{k\times l},
\end{equation*}
and
\begin{equation*}
  \overleftarrow{B}(\omega,t):=\int_t^T \phi(s)dB(\omega,s),
\end{equation*}
with $\phi\in L^{2}([0,T];\mathbb{R}^{l})$. Here the integral with respect to $\{ B(\omega,s) \}$ is a ``backward It\^{o} integral''.
It is easy to see that $\overleftarrow{B}_{t}$ is $\mathcal{F}^{B}_{t,T}$-measurable, for any $t\in[0,T].$
The coefficients $f$ and $g$ can be seen as the compound functions of  $\overline{f}$ and $\overleftarrow{B}$,
 and $\overline{g}$ and $\overleftarrow{B}$, respectively.

We assume that for any $e\in \mathbb{R}^{l}$, $(x,y,z)\rightarrow ((\overline{f}(e,x,y,z),(\overline{g}(e,x,y,z))$ is of class $C^{3}$,
and all derivatives are bounded on $\mathbb{R}^{l}\times \mathbb{R}^{d}\times \mathbb{R}^{k}\times \mathbb{R}^{k\times d}.$
Thus for any $(\omega,t)\in \Omega\times [0,T]$,  $(x,y,z)\rightarrow (f(t,x,y,z),(g(t,x,y,z))$ is also of class $C^{3}$,
and all derivatives are bounded too.

We assume again that (H1) and (H2) hold, together with
\begin{equation*} (H3): \ \ \ \
    \overline{g}'_{z}(e,x,y,z) \theta \theta^{\ast}  \overline{g}'_{z}(e,x,y,z)^{\ast} \leq  \theta \theta^{\ast},  \q
    \forall  (e,x,y,z)\in \mathbb{R}^{l}\times \mathbb{R}^{d}\times \mathbb{R}^{k}\times \mathbb{R}^{k\times d}.
\end{equation*}
Let $h\in C_{p}^{3}(\mathbb{R}^{d};\mathbb{R}^{k})$. For any $(\omega,t)\in\Omega\times [0,T], \ x\in \mathbb{R}^{d}$,
denote by $\{ (Y^{t,x}_{s},Z^{t,x}_{s});t\leq s\leq T, x\in\mathbb{R}^{d} \}$ the unique solution of BDSDE:
\begin{equation}\label{2}
 \begin{split}
  Y^{t,x}_{s} =& h(X^{t,x}_{T}) + \int_s^T f(\omega,r,X^{t,x}_{r},Y^{t,x}_{r},Z^{t,x}_{r}) dr \\
               &+ \int_s^T g(\omega,r,X^{t,x}_{r},Y^{t,x}_{r},Z^{t,x}_{r}) dB_{r} - \int_s^T Z^{t,x}_{r}dW_{r}, \ \ t\leq s\leq T.
 \end{split}
\end{equation}
From the relation (\ref{12}), BDSDE (\ref{2}) can be rewritten as
\begin{equation*}
 \begin{split}
  Y^{t,x}_{s} =& h(X^{t,x}_{T}) + \int_s^T \overline{f}(\overleftarrow{B}(\omega,t),X^{t,x}_{r},Y^{t,x}_{r},Z^{t,x}_{r}) dr \\
               &+ \int_s^T \overline{g}(\overleftarrow{B}(\omega,t),X^{t,x}_{r},Y^{t,x}_{r},Z^{t,x}_{r}) dB_{r}-\int_s^T Z^{t,x}_{r}dW_{r}, \ \ t\leq s\leq T.
 \end{split}
\end{equation*}
We shall define $X^{t,x}_{s}, \ Y^{t,x}_{s}$ and $Z^{t,x}_{s}$ for all $(s,t)\in [0,T]^{2}$ by letting
$X^{t,x}_{s}=X^{t,x}_{s\vee t}, \ Y^{t,x}_{s}=Y^{t,x}_{s\vee t}$ and $Z^{t,x}_{s}=0$ for $s<t$.

\begin{theorem}\label{6}
$\{Y^{t,x}_{s}; (t,s)\in [0,T]^{2}, x\in\mathbb{R}^{d} \}$ has a version whose trajectories belong to $C^{0,0,2}([0,T]^{2}\times \mathbb{R}^{d})$.
\end{theorem}

Before proceeding to the proof of this theorem, let us state an immediate corollary:

\begin{corollary}\label{18}
There exists a continuous version of the random field $\{ Y^{t,x}_{t};$ $ t\in [0,T], x\in \mathbb{R}^{d} \}$ such that for any $t\in[0,T]$,
$x\rightarrow Y^{t,x}_{t}$ is of class $C^{2}$ a.s., the derivatives being a.s. continuous in $(t,x)$.
\end{corollary}

\begin{proof}[Proof of Theorem \ref{6}]
We point out that the proof is similar to the proof of Theorem 2.1 of Pardoux and Peng \cite{Peng}.
We first note that we can deduce from Proposition \ref{8} applied to the present situation that,
for any $p\geq 2$, there exist $c_{p}$ and $q$ such that:
\begin{equation*}
  E\bigg[\sup_{t\leq s\leq T} | Y^{t,x}_{s}|^{p} + \left(\int_t^T \|  Z^{t,x}_{s} \|^{2}ds\right)^{\frac{p}{2}} \bigg]\leq c_{p} (1+ |x|^{q}).
\end{equation*}
Next, for $t\vee t' \leq s\leq T$,
\begin{equation*}
 \begin{split}
  Y^{t,x}_{s} - Y&^{t',x'}_{s}= \left[ \int_0^1 h'(X^{t',x'}_{T} + \lambda(X^{t,x}_{T}-X^{t',x'}_{T})) d\lambda \right] [X^{t,x}_{T}-X^{t',x'}_{T}]\\
   &\qq \ \ + \int_s^T \bigg[ \varphi_{r}(t,x;t',x')[X^{t,x}_{r}-X^{t',x'}_{r}] + \psi_{r}(t,x;t',x')[Y^{t,x}_{r}-Y^{t',x'}_{r}]\\
   &\qq \ \ \ \ \ \ \ \ \ \ \ \ \ + \chi_{r}(t,x;t',x')[Z^{t,x}_{r}-Z^{t',x'}_{r}]  \bigg] dr\\
   &\qq \ \ + \int_s^T \bigg[ \overline{\varphi}_{r}(t,x;t',x')[X^{t,x}_{r}-X^{t',x'}_{r}] + \overline{\psi}_{r}(t,x;t',x')[Y^{t,x}_{r}-Y^{t',x'}_{r}]\\
   &\qq \ \ \ \ \ \ \ \ \ \ \ \ \ + \overline{\chi}_{r}(t,x;t',x')[Z^{t,x}_{r}-Z^{t',x'}_{r}]  \bigg] dB_{r} - \int_s^T [Z^{t,x}_{r}-Z^{t',x'}_{r}] dW_{r},
 \end{split}
\end{equation*}
where
\begin{equation*}
 \begin{split}
   \varphi_{r}(t,x;t',x') &= \int_0^1 \overline{f}'_{x}(\Sigma_{r,\lambda}^{t,x;t',x'}) d\lambda;\\
      \psi_{r}(t,x;t',x') &= \int_0^1 \overline{f}'_{y}(\Sigma_{r,\lambda}^{t,x;t',x'}) d\lambda;\\
      \chi_{r}(t,x;t',x') &= \int_0^1 \overline{f}'_{z}(\Sigma_{r,\lambda}^{t,x;t',x'}) d\lambda,
 \end{split}
\end{equation*}
$\overline{\varphi}, \ \overline{\psi}$ and $\overline{\chi}$ are defined analogously, with $\overline{f}$ replaced by $\overline{g}$, and
\begin{equation*}
   \Sigma_{r,\lambda}^{t,x;t',x'} = \big(\overleftarrow{B}_{r},X^{t',x'}_{r} + \lambda(X^{t,x}_{r}-X^{t',x'}_{r}),
    Y^{t',x'}_{r} + \lambda(Y^{t,x}_{r}-Y^{t',x'}_{r}), Z^{t',x'}_{r} + \lambda(Z^{t,x}_{r}-Z^{t',x'}_{r}) \big).
 \end{equation*}
Combining the argument of Proposition \ref{8} with the estimate:
\begin{equation*}
  E\big[\sup_{0\leq s\leq T} | X^{t,x}_{s} - X^{t',x'}_{s}|^{p} \big]\leq c_{p} (1+ |x|^{q}+|x'|^{q}) (|x-x'|^{p} + |t-t'|^{\frac{p}{2}}),
\end{equation*}
we deduce that for all $p\geq 2$, there exists $c_{p}$ and $q$ such that
\begin{equation*}
 \begin{split}
   &E\bigg[\sup_{0\leq s\leq T} | Y^{t,x}_{s} - Y^{t',x'}_{s}|^{p}
   + \left(\int_t^T \|  Z^{t,x}_{s} -  Z^{t',x'}_{s} \|^{2}ds\right)^{\frac{p}{2}} \bigg]\\
   \leq& c_{p} (1+ |x|^{q}+|x'|^{q}) (|x-x'|^{p} + |t-t'|^{\frac{p}{2}}).
 \end{split}
\end{equation*}
Note that (H1) implies that $\|\overline{\chi}_{r}\| \leq \alpha<1$.
We conclude from the last estimate, using Kolmogorov's lemma,
that $\{Y^{t,x}_{s}; (t,s)\in [0,T]^{2}, x\in\mathbb{R}^{d} \}$ has an a.s. continuous version.

Next, we define
\begin{equation*}
  \Delta_{h}^{i} X^{t,x}_{s} := \frac{X^{t,x+h_{e_{i}}}_{s}-X^{t,x}_{s}}{h},
\end{equation*}
where $h\in \mathbb{R}\ \{0\}$, $\{ e_{1},e_{2},...,e_{3} \}$ is an orthonormal basis of $\mathbb{R}^{d}$.
$\Delta_{h}^{i} Y^{t,x}_{s}$ and $\Delta_{h}^{i} Z^{t,x}_{s}$ are defined analogously.
We have
\begin{equation*}
  \begin{split}
\Delta_{h}^{i}&Y^{t,x}_{s}= \int_0^1 h'(X^{t,x}_{T} + \lambda h \Delta_{h}^{i} X^{t,x}_{T}) \Delta_{h}^{i} X^{t,x}_{T} d\lambda
     - \int_s^T \Delta_{h}^{i} Z^{t,x}_{r} dW_{r}\\
    &\qq \ \ + \int_s^T\int_0^1 \big[ \overline{f}'_{x}(\Xi_{r,\lambda}^{t,x,h})\Delta_{h}^{i} X^{t,x}_{r}
     + \overline{f}'_{y}(\Xi_{r,\lambda}^{t,x,h})\Delta_{h}^{i} Y^{t,x}_{r}
     + \overline{f}'_{z}(\Xi_{r,\lambda}^{t,x,h})\Delta_{h}^{i} Z^{t,x}_{r} \big] d\lambda dr\\
    &\qq \ \ + \int_s^T\int_0^1 \big[ \overline{g}'_{x}(\Xi_{r,\lambda}^{t,x,h})\Delta_{h}^{i} X^{t,x}_{r}
     + \overline{g}'_{y}(\Xi_{r,\lambda}^{t,x,h})\Delta_{h}^{i} Y^{t,x}_{r}
     + \overline{g}'_{z}(\Xi_{r,\lambda}^{t,x,h})\Delta_{h}^{i} Z^{t,x}_{r} \big] d\lambda dB_{r},\\
  \end{split}
\end{equation*}
where
$ \Xi_{r,\lambda}^{t,x,h} = \big(\overleftarrow{B}_{r}, X^{t,x}_{r} + \lambda h \Delta_{h}^{i} X^{t,x}_{r},
    Y^{t,x}_{r} + \lambda h \Delta_{h}^{i} Y^{t,x}_{r}, Z^{t,x}_{r} + \lambda h \Delta_{h}^{i} Z^{t,x}_{r} \big).$
We note that for each $p \geq 2$, there exists $c_{p}$ such that
\begin{equation*}
  E\big[\sup_{0\leq s\leq T} |\Delta_{h}^{i} X^{t,x}_{s}|^{p} \big]\leq c_{p}.
\end{equation*}
The same estimates as above yields
\begin{equation*}
  E\bigg[\sup_{t\leq s\leq T} |\Delta_{h}^{i} Y^{t,x}_{s}|^{p} + \left(\int_t^T \| \Delta_{h}^{i} Z^{t,x}_{s} \|^{2}ds\right)^{\frac{p}{2}} \bigg]
  \leq c_{p}(1 + |x|^{q} + |h|^{q}).
\end{equation*}
Finally, we consider
\begin{equation*}
  \begin{split}
\Delta_{h}^{i}&Y^{t,x}_{s} -\Delta_{h'}^{i}Y^{t',x'}_{s}
  = - \int_s^T \big[ \Delta_{h}^{i} Z^{t,x}_{r} -  \Delta_{h'}^{i} Z^{t',x'}_{r} \big] dW_{r}\\
    &\qq + \int_0^1 \big[h'(X^{t,x}_{T} + \lambda h \Delta_{h}^{i} X^{t,x}_{T}) \Delta_{h}^{i} X^{t,x}_{T}
                     -h'(X^{t',x'}_{T} + \lambda h' \Delta_{h'}^{i} X^{t',x'}_{T}) \Delta_{h'}^{i} X^{t',x'}_{T} \big]d\lambda\\
    &\qq + \int_s^T\int_0^1 \big[ \overline{f}'_{x}(\Xi_{r,\lambda}^{t,x,h})\Delta_{h}^{i} X^{t,x}_{r}
                             -\overline{f}'_{x}(\Xi_{r,\lambda}^{t',x',h'})\Delta_{h'}^{i} X^{t',x'}_{r} \big] d\lambda dr\\
    &\qq + \int_s^T\int_0^1 \big[ \overline{f}'_{y}(\Xi_{r,\lambda}^{t,x,h})\Delta_{h}^{i} Y^{t,x}_{r}
                             -\overline{f}'_{y}(\Xi_{r,\lambda}^{t',x',h'})\Delta_{h'}^{i} Y^{t',x'}_{r} \big] d\lambda dr\\
    &\qq + \int_s^T\int_0^1 \big[ \overline{f}'_{z}(\Xi_{r,\lambda}^{t,x,h})\Delta_{h}^{i} Z^{t,x}_{r}
                             -\overline{f}'_{z}(\Xi_{r,\lambda}^{t',x',h'})\Delta_{h'}^{i} Z^{t',x'}_{r} \big] d\lambda dr\\
    &\qq + \int_s^T\int_0^1 \big[ \overline{g}'_{x}(\Xi_{r,\lambda}^{t,x,h})\Delta_{h}^{i} X^{t,x}_{r}
                             -\overline{g}'_{x}(\Xi_{r,\lambda}^{t',x',h'})\Delta_{h'}^{i} X^{t',x'}_{r} \big] d\lambda dB_{r}\\
    &\qq + \int_s^T\int_0^1 \big[ \overline{g}'_{y}(\Xi_{r,\lambda}^{t,x,h})\Delta_{h}^{i} Y^{t,x}_{r}
                             -\overline{g}'_{y}(\Xi_{r,\lambda}^{t',x',h'})\Delta_{h'}^{i} Y^{t',x'}_{r} \big] d\lambda dB_{r}\\
    &\qq + \int_s^T\int_0^1 \big[ \overline{g}'_{z}(\Xi_{r,\lambda}^{t,x,h})\Delta_{h}^{i} Z^{t,x}_{r}
                             -\overline{g}'_{z}(\Xi_{r,\lambda}^{t',x',h'})\Delta_{h'}^{i} Z^{t',x'}_{r} \big] d\lambda dB_{r}.
  \end{split}
\end{equation*}
We note that
\begin{equation*}
  E\big[\sup_{0\leq s\leq T} |\Delta_{h}^{i} X^{t,x}_{r}-\Delta_{h'}^{i} X^{t',x'}_{r}|^{p}\big]
   \leq c_{p}(1+ |x|^{q}+|x'|^{q})(|x-x'|^{p} + |h-h'|^{p} + |t-t'|^{\frac{p}{2}}),
\end{equation*}
and
\begin{equation*}
  \begin{split}
    |\Xi_{r,\lambda}^{t,x,h} - \Xi_{r,\lambda}^{t',x',h'}| \leq& \big( |X^{t,x}_{r}-X^{t',x'}_{r}| + |X^{t,x+h_{e_{i}}}_{r}-X^{t',x'+h'_{e_{i}}}_{r}|\\
     & + |Y^{t,x}_{r}-Y^{t',x'}_{r}| + |Y^{t,x+h_{e_{i}}}_{r}-Y^{t',x'+h'_{e_{i}}}_{r}|\\
     & + \|Z^{t,x}_{r}-Z^{t',x'}_{r}\| + \|Z^{t,x+h_{e_{i}}}_{r}-Z^{t',x'+h'_{e_{i}}}_{r}\| \big).
  \end{split}
\end{equation*}
Using similar arguments as those in Proposition \ref{8}, combined with those of Theorem 2.9 in Pardoux and Peng \cite{Peng92}, we show that
\begin{equation*}
 \begin{split}
   &E\bigg[\sup_{0\leq s\leq T} | \Delta_{h}^{i} Y^{t,x}_{s} - \Delta_{h'}^{i} Y^{t',x'}_{s}|^{p}
   + \left(\int_{t\vee t'}^T \|  \Delta_{h}^{i}Z^{t,x}_{s} -  \Delta_{h'}^{i}Z^{t',x'}_{s} \|^{2}ds\right)^{\frac{p}{2}} \bigg]\\
   \leq& c_{p} (1+|x|^{q}+|x'|^{q}+|h|^{q}+|h'|^{q}) \times (|x-x'|^{p} + |h-h'|^{p} + |t-t'|^{\frac{p}{2}}).
 \end{split}
\end{equation*}
The existence of a continuous derivative of $Y^{t,x}_{s}$ with respect to $x$ follows easily from the above estimate,
as well as the existence of a mean-square derivative of $Z^{t,x}_{s}$ with respect to $x$,
which is mean square continuous in $(s,t,x)$.
The existence of a continuous second derivative of $Y^{t,x}_{s}$ with respect to $x$ is proved in a similar fashion.
\end{proof}

From Lemma 2.3 of Pardoux-Peng \cite{Peng92}, we can get the following result similarly.

\begin{lemma}\label{19}
Let $Z\in M^{2}(t,T;\mathbb{R}^{d})$ be such that $\xi:=\int_t^T Z_{r} dW_{r}$ satisfies $\xi \in \mathbb{D}^{1,2}$.
Then
\begin{equation*}\label{}
  Z_{i}\in L^{2}(t,T;\mathbb{D}^{1,2}), \ \ 1\leq i\leq d,
\end{equation*}
and for $t\leq s \leq T$,
\begin{equation*}\label{}
  D_{s}^{i}\xi = (Z_{s})_{i} + \int_s^T  D_{s}^{i} Z_{r}dW_{r},
\end{equation*}
where $(Z_{s})_{i}$ denotes the $i$-th component of $Z_{s}$.
\end{lemma}

It turns out that under the assumptions of Theorem \ref{6}, the components of $X^{t,x}_{s}, \ Y^{t,x}_{s}$ and $Z^{t,x}_{s}$ take values in $\mathbb{D}^{1,2}$.
From Pardoux and  Peng \cite{Peng92}, we have the next formula,
\begin{equation}\label{7}
  D_{\theta}X^{t,x}_{s} = \nabla X^{t,x}_{s} (\nabla X^{t,x}_{\theta})^{-1}\sigma(X^{t,x}_{\theta}), \ \ \ t\leq \theta\leq s\leq T.
\end{equation}

Dropping the superscript $t,x$ for convenience, let us now express $Z$ in terms of the Wiener space derivative of $Y$.
\begin{proposition}\label{10}
Under the assumptions of Theorem \ref{6}, $Y,Z\in L^{2}(t,T;\mathbb{D}^{1,2})$,
 and a version of $\{ (D_{\theta}Y_{s},D_{\theta}Z_{s}); t\leq \theta\leq T; t\leq s\leq T \}$ is given by:\\

(i) \ \ $D_{\theta}Y_{s}=0, D_{\theta}Z_{s}=0;$ for $ t\leq s< \theta\leq T$;

(ii) \ For any fixed $\theta\in [t,T]$ and $1\leq i\leq d,$ $\{(D^{i}_{\theta}Y_{s},D^{i}_{\theta}Z_{s}); t\leq \theta\leq s\leq T \}$
is the unique solution of BDSDE:
\begin{equation}\label{9}
 \begin{split}
   D^{i}_{\theta}Y_{s} =& h'(X_{T})D^{i}_{\theta}X_{T} +  \int_s^T F_{i}(\omega,r,D^{i}_{\theta}X_{r}, D^{i}_{\theta}Y_{r}, D^{i}_{\theta}Z_{r}) dr \\
               &+ \int_s^T G_{i}(\omega,r,D^{i}_{\theta}X_{r}, D^{i}_{\theta}Y_{r}, D^{i}_{\theta}Z_{r}) dB_{r} - \int_s^T D^{i}_{\theta}Z_{r} dW_{r},
 \end{split}
\end{equation}
where
\begin{equation}\label{15}
 \begin{split}
   F_{i}(\omega,r,x, y, z) =& f'_{x}(\omega,r,X_{r},Y_{r},Z_{r})x+f'_{y}(\omega,r,X_{r},Y_{r},Z_{r})y \\
   &+ f'_{z}(\omega,r,X_{r},Y_{r},Z_{r})z,\\
   G_{i}(\omega,r,x, y, z) =& g'_{x}(\omega,r,X_{r},Y_{r},Z_{r})x+g'_{y}(\omega,r,X_{r},Y_{r},Z_{r})y \\
   &+ g'_{z}(\omega,r,X_{r},Y_{r},Z_{r})z.\\
 \end{split}
\end{equation}
Moreover, $\{D^{i}_{s}Y_{s} ; t\leq s\leq T \}$ is a version of $\{(Z_{s})_{i} ; t\leq s\leq T \}$,
where $(Z_{s})_{i}$ denote the $i$-th column of the matrix $Z$.
\end{proposition}

\begin{proof}
We restrict ourselves to the case $d=1$, and note that $f'_{x}(\omega,r,X_{r},Y_{r},$ $Z_{r})$, $f'_{y}(\omega,r,X_{r},Y_{r},Z_{r})$
and $f'_{z}(\omega,r,X_{r},Y_{r},Z_{r})$ are bounded in $L^{2}(\Omega\times [0,T])$.

The item (i) is trivial consequence of the fact that $Y_{s}$ and $Z_{s}$ are $\mathcal{F}_{s}$-measurable.
Now we consider the item (ii).
From Eq. (\ref{2}) and the definition (\ref{0}), note that $t\leq \theta\leq s\leq T$, we deduce
\begin{equation*}\label{}
 \begin{split}
   D^{i}_{\theta}Y_{s} =& h'(X_{T})D^{i}_{\theta}X_{T}
      +  \int_s^T D^{i}_{\theta}f(\omega,r,X_{r},Y_{r},Z_{r}) dr \\
     &+ \int_s^T D^{i}_{\theta}g(\omega,r,X_{r},Y_{r},Z_{r}) dB_{r} - \int_s^T D^{i}_{\theta}Z_{r} dW_{r},
 \end{split}
\end{equation*}
where
\begin{equation*}
 \begin{split}
     &D^{i}_{\theta}f(\omega,r,X_{r},Y_{r},Z_{r})
   = D^{i}_{\theta}\overline{f}(\overleftarrow{B}_{r},X_{r},Y_{r},Z_{r})\\
   =& \overline{f}_{x}'(\overleftarrow{B}_{r},X_{r},Y_{r},Z_{r})D^{i}_{\theta}X_{r}
     +\overline{f}_{y}'(\overleftarrow{B}_{r},X_{r},Y_{r},Z_{r})D^{i}_{\theta}Y_{r}
     +\overline{f}_{z}'(\overleftarrow{B}_{r},X_{r},Y_{r},Z_{r})D^{i}_{\theta}Z_{r} \\
   =& f_{x}'(\omega,r,X_{r},Y_{r},Z_{r})D^{i}_{\theta}X_{r}
     +f_{y}'(\omega,r,X_{r},Y_{r},Z_{r})D^{i}_{\theta}Y_{r}
     +f_{z}'(\omega,r,X_{r},Y_{r},Z_{r})D^{i}_{\theta}Z_{r},
 \end{split}
\end{equation*}
and similarly,
\begin{equation*}
 \begin{split}
     &D^{i}_{\theta}g(\omega,r,X_{r},Y_{r},Z_{r})\\
   =& g_{x}'(\omega,r,X_{r},Y_{r},Z_{r})D^{i}_{\theta}X_{r}
     +g_{y}'(\omega,r,X_{r},Y_{r},Z_{r})D^{i}_{\theta}Y_{r}
     +g_{z}'(\omega,r,X_{r},Y_{r},Z_{r})D^{i}_{\theta}Z_{r}.
 \end{split}
\end{equation*}
Hence we have
\begin{equation*}
 \begin{split}
   D^{i}_{\theta}Y_{s} =& h'(X_{T})D^{i}_{\theta}X_{T} +  \int_s^T F_{i}(\omega,r,D^{i}_{\theta}X_{r}, D^{i}_{\theta}Y_{r}, D^{i}_{\theta}Z_{r}) dr \\
               &+ \int_s^T G_{i}(\omega,r,D^{i}_{\theta}X_{r}, D^{i}_{\theta}Y_{r}, D^{i}_{\theta}Z_{r}) dB_{r} - \int_s^T D^{i}_{\theta}Z_{r} dW_{r}.
 \end{split}
\end{equation*}
Now the existence and uniqueness of solutions of above equation follows easily from the results of Pardoux and Peng \cite{Peng90}
by using the same kind of method, i.e., the Picard iteration.

Finally, since for $t<\theta\leq s\leq T$,
\begin{equation*}
 \begin{split}
  Y_{s} =& Y_{t} - \int_t^s f(\omega,r,X_{r},Y_{r},Z_{r}) dr
                 - \int_t^s g(\omega,r,X_{r},Y_{r},Z_{r}) dB_{r} + \int_t^s Z_{r}dW_{r}, \\
  D^{i}_{\theta} Y_{s}=&(Z_{\theta})_{i} + \int_\theta^s D^{i}_{\theta}Z_{r} dW_{r}\\
     &- \int_\theta^s \big[ f_{x}'(\omega,r,X_{r},Y_{r},Z_{r})D^{i}_{\theta}X_{r} +f_{y}'(\omega,r,X_{r},Y_{r},Z_{r})D^{i}_{\theta}Y_{r}\\
     & \ \ \ \ \ \ \ \ \    +f_{z}'(\omega,r,X_{r},Y_{r},Z_{r})D^{i}_{\theta}Z_{r}\big] dr\\
     &- \int_\theta^s \big[ g_{x}'(\omega,r,X_{r},Y_{r},Z_{r})D^{i}_{\theta}X_{r} +g_{y}'(\omega,r,X_{r},Y_{r},Z_{r})D^{i}_{\theta}Y_{r}\\
     & \ \ \ \ \ \ \ \ \    +g_{z}'(\omega,r,X_{r},Y_{r},Z_{r})D^{i}_{\theta}Z_{r}\big] dB_{r},
 \end{split}
\end{equation*}
for a.e. $s$, the jump of $D_{\theta} Y_{s}$ at $\theta=s$ equals $Z_{s}$.
With the version of $D_{\theta} Y_{s}$ that we have chosen above, that means exactly that
\begin{equation*}
  D_{s}Y_{s}=Z_{s}, \ \ \ a.e.
\end{equation*}
Our proof is completed.
\end{proof}

We next want to show that $\{D_{s}Y_{s} ; t\leq s\leq T \}$ processes an a.s. continuous version.
It is easy to deduce, as in Pardoux and Peng \cite{Peng92}, that
 $\big\{ \big(\nabla Y_{s}=\frac{\partial Y^{t,x}_{s}}{\partial x},\nabla Z_{s}=\frac{\partial Z^{t,x}_{s}}{\partial x}\big) \big\}$
is the unique solution of BDSDE:
\begin{equation*}
 \begin{split}
   \nabla Y_{s} =& h'(X_{T})\nabla X_{T} +  \int_s^T F(\omega,r, \nabla X_{r},\nabla Y_{r}, \nabla Z_{r}) dr \\
               &+ \int_s^T G(\omega,r, \nabla X_{r}, \nabla Y_{r}, \nabla Z_{r}) dB_{r} - \int_s^T \nabla Z_{r} dW_{r},
 \end{split}
\end{equation*}
where $F$ and $G$ are defined  in (\ref{15}).

From the uniqueness of the solutions of BDSDE (\ref{9}) and the formula (\ref{7}), we directly have
\begin{equation}\label{11}
  D_{\theta}Y_{s} = \nabla Y_{s} (\nabla X_{\theta})^{-1}\sigma(X_{\theta}), \ \ \ t\leq \theta\leq s\leq T.
\end{equation}
and the process $\{D_{s}Y_{s} ; t\leq s\leq T \}$ as defined by Proposition \ref{10} is a.s. continuous by virtue of the continuity
of $\nabla Y_{s},\nabla X_{s}$ and $X_{s}$.

Now an immediate consequence of Proposition \ref{10} and Eq. (\ref{11}) is:
\begin{proposition}\label{16}
The random field $\{Z^{t,x}_{s}; t\leq s\leq T,x\in\mathbb{R}^{d} \}$  has an a.s. continuous version which is given by:
\begin{equation*}
  Z^{t,x}_{s} = \nabla Y^{t,x}_{s} (\nabla X^{t,x}_{\theta})^{-1}\sigma(X^{t,x}_{\theta}),
\end{equation*}
and in particular
\begin{equation*}
  Z^{t,x}_{t} = \nabla Y^{t,x}_{t}\sigma(x).
\end{equation*}
\end{proposition}

\section{BDSDEs with random coefficients and quasilinear SPDEs}

We now relate our BDSDE to the following system of quasilinear backward stochastic partial differential equations:

\begin{equation}\label{3}
 \begin{split}
  u(t,x)=& h(x) + \int_t^T \big[\mathcal{L}u(s,x) + f(\omega,s,x,u(s,x),(\nabla u \sigma)(s,x))\big] ds \\
         & + \int_t^T g(\omega,s,x,u(s,x),(\nabla u \sigma)(s,x)) dB_{s}, \ \ 0\leq t\leq T,
 \end{split}
\end{equation}
where $u:\mathbb{R}_{+}\times \mathbb{R}^{d}\rightarrow  \mathbb{R}^{k}$ and $\mathcal{L}u=(Lu_{1},...,Lu_{k})^{\ast}$ with
\begin{equation*}
  L=\frac{1}{2}\sum_{i,j=1}^{d} (\sigma \sigma^{\ast})_{ij}(t,x)\frac{\partial^{2}}{\partial x_{i}\partial x_{j}}
    + \sum_{i=1}^{d} b_{i}(t,x) \frac{\partial}{\partial x_{i}}.
\end{equation*}

\begin{theorem}\label{13}
Let $f$ and $g$ satisfy the assumptions of Sects. 2 and 3, and $h$ be of class $C^{2}(\mathbb{R}^{d})$.
Let $\{u(t,x); 0\leq t\leq T, x\in\mathbb{R}^{d}\}$ be a random field such that $u(t,x)$ is $\mathcal{F}_{t,T}^{B}$-measurable and for each $(t,x)$,
$u(t,x)\in C^{0,2}([0,T]\times\mathbb{R}^{d})$ a.s., and $u$ satisfies SPDE (\ref3).
Then $u(t,x)=Y_{t}^{t,x}$, where $\{ (Y_{s}^{t,x},Z_{s}^{t,x}); 0\leq t\leq s\leq T, x\in\mathbb{R}^{d} \}$  is the unique solution of BDSDE (\ref{2}).
\end{theorem}

\begin{proof}
It suffices to show that $\big\{\big(u(s,X^{t,x}_{s}),(\nabla u \sigma)(s,X^{t,x}_{s})\big); 0\leq t\leq s\leq T, x\in\mathbb{R}^{d}\big\}$ solves BDSDE (\ref{2}).

Let $t=t_{0}< t_{1}< . . . <t_{n}=T$, then
\begin{equation*}
 \begin{split}
   & \sum_{i=0}^{n-1} \big[ u(t_{i},X^{t,x}_{t_{i}}) - u(t_{i+1},X^{t,x}_{t_{i+1}}) \big] \\
  =& \sum_{i=0}^{n-1} \big[ u(t_{i},X^{t,x}_{t_{i}}) - u(t_{i},X^{t,x}_{t_{i+1}}) \big]
    +\sum_{i=0}^{n-1} \big[ u(t_{i},X^{t,x}_{t_{i+1}}) - u(t_{i+1},X^{t,x}_{t_{i+1}}) \big]\\
  =& \sum_{i=0}^{n-1} \bigg[ -\int_{t_{i}}^{t_{i+1}} \mathcal{L}u(t_{i},X^{t,x}_{s}) ds-\int_{t_{i}}^{t_{i+1}} (\nabla u \sigma)(t_{i},X^{t,x}_{s}) dW_{s} \\
   & \ \ \ \ \ \ \ + \int_{t_{i}}^{t_{i+1}} \big[ \mathcal{L}u(s,X^{t,x}_{t_{i+1}})
    + f(\omega,s,X^{t,x}_{t_{i+1}},u(s,X^{t,x}_{t_{i+1}}),(\nabla u \sigma)(s,X^{t,x}_{t_{i+1}})) \big] ds\\
   & \ \ \ \ \ \ \ + \int_{t_{i}}^{t_{i+1}} g(\omega,s,X^{t,x}_{t_{i+1}},u(s,X^{t,x}_{t_{i+1}}),(\nabla u \sigma)(s,X^{t,x}_{t_{i+1}})) dB_{s}  \bigg],
 \end{split}
\end{equation*}
where we have used the It\^{o} formula and the equation satisfied by $u$.
It finally suffices to let the mesh size go to zero in order to conclude.
\end{proof}

We are now in a position to prove the converse to the above result:
\begin{theorem}\label{14}
Let $f, \ g$ and $h$ satisfy the assumptions of Sect. 2 and 3.
Then $\{ u(t,x):=Y_{t}^{t,x}, 0\leq t\leq T, x\in\mathbb{R}^{d} \}$ is the unique $C^{0,2}([0,T]\times\mathbb{R}^{d})$-solution of the quasilinear backward parabolic SPDE (\ref{3}).
\end{theorem}

\begin{proof}
From Theorem \ref{6},  $u\in C^{0,2}([0,T]\times\mathbb{R}^{d})$.
Let $h>0$ be such that $t+h\leq T$.
We first note that $Y^{t,x}_{t+h}=Y^{t+h,X^{t,x}_{t+h}}_{t+h}$.
Hence
\begin{equation*}
 \begin{split}
   & u(t+h,x)- u(t,x)\\
  =&  [u(t+h,x)- u(t+h,X^{t,x}_{t+h})] + [ u(t+h,X^{t,x}_{t+h}) - u(t,x) ] \\
%\ \ \ \ \ \ \ \ \ \ \ \ \ \ \ \ \ \ \ \ \ \ \
% \end{split}
%\end{equation*}
%\begin{equation*}
% \begin{split}
  =& -\int_{t}^{t+h} \mathcal{L}u(t+h,X^{t,x}_{s}) ds - \int_{t}^{t+h} (\nabla u \sigma)(t+h,X^{t,x}_{s}) dW_{s} +\int_{t}^{t+h} Z^{t,x}_{s} dW_{s}\\
   & -\int_{t}^{t+h} f(\omega,s,X^{t,x}_{s},Y^{t,x}_{s},Z^{t,x}_{s}) ds -\int_{t}^{t+h} g(\omega,s,X^{t,x}_{s},Y^{t,x}_{s},Z^{t,x}_{s}) dB_{s}, \\
 \end{split}
\end{equation*}
where we have used the It\^{o} formula and the BDSDE.
Now let $t=t_{0}< t_{1}< . . . <t_{n}=T$. We have
\begin{equation*}
 \begin{split}
   h(x)- u(t,x)
  =& \sum_{i=0}^{n-1} \bigg[ -\int_{t_{i}}^{t_{i}+1} \big[\mathcal{L}u(t+h,X^{t,x}_{s}) + f(\omega,s,X^{t,x}_{s},Y^{t,x}_{s},Z^{t,x}_{s})\big] ds \\
   & \ \ \ \ \ \ \ +\int_{t_{i}}^{t_{i}+1} \big[Z^{t,x}_{s} - (\nabla u \sigma)(t+h,X^{t,x}_{s})\big] dW_{s}\\
   & \ \ \ \ \ \ \ -\int_{t_{i}}^{t_{i}+1} g(\omega,s,X^{t,x}_{s},Y^{t,x}_{s},Z^{t,x}_{s}) dB_{s} \bigg]. \\
 \end{split}
\end{equation*}
It follows from Theorem \ref{6} and Proposition \ref{16} that if we let the mesh size go to zero, we obtain in the limit:
\begin{equation*}
 \begin{split}
  u(t,x) =& h(x) + \int_t^T \big[\mathcal{L}u(s,x) + f(\omega,s,x,u(s,x),(\nabla u \sigma)(s,x))\big] ds \\
          &+ \int_t^T g(\omega,s,x,u(s,x),(\nabla u \sigma)(s,x)) dB_{s}.
 \end{split}
\end{equation*}
Hence $u\in C^{1,2}([0,T]\times\mathbb{R}^{d})$ and satisfies SPDE (\ref{3}).
\end{proof}

\begin{remark}
Our result extends the nonlinear stochastic Feynman-Kac formula of Pardoux-Peng \cite{Peng}
(and the linear stochastic Feynman-Kac formula of Pardoux \cite{Pardoux})
to non-Markovian case.
Indeed, if $f$ and $g$ are determined, i.e., independent of $\omega$,
then our stochastic Feynman-Kac formula degenerates to the result of Pardoux-Peng \cite{Peng}.
\end{remark}

\begin{remark}
Our result also expands the nonlinear Feynman-Kac formula of Pardoux-Peng \cite{Peng92} to the random coefficient case.
In fact, if the coefficient $g\equiv 0$, then BDSDE (\ref{2}) degenerates to a BSDE,
and our stochastic Feynman-Kac formula extends the result of Pardoux-Peng \cite{Peng} to non-Markovian case.
In this case, a similar (but different) result was obtained in Ma et al. \cite{Zhang}.
\end{remark}

\begin{remark}
Let $\varphi\in M^{2}(0,T; \mathbb{R}^{l})$ such that $\varphi_{t}$ takes value in $\mathbb{D}^{1,2}$ for all $t\in[0,T]$. Denote
\begin{equation*}
  \psi (\omega,t):=E\big[\varphi(\omega,t)|\mathcal{F}^{B}_{t,T}\big].
\end{equation*}
Then, in (\ref{12}), if the coefficients $f$ and $g$ are of the form
\begin{equation*}\label{}
 \begin{split}
   f(\omega,t,x,y,z)&= \overline{f}(\psi(\omega,t),x,y,z),\\
   g(\omega,t,x,y,z)&= \overline{g}(\psi(\omega,t),x,y,z),\\
 \end{split}
\end{equation*}
Theorem \ref{13} and Theorem \ref{14} can also be established by the similar way.
\end{remark}

%%\section*{Acknowledgments}

\bibliographystyle{elsarticle-num}

\end{document}